\documentclass[12pt,reqno]{amsart}
 
\usepackage{latexsym, amsmath, amssymb, amscd}

\usepackage{mathrsfs}

\newtheorem{theorem}{Theorem}[section]
\newtheorem{lemma}[theorem]{Lemma}
\newtheorem{proposition}[theorem]{Proposition}

\theoremstyle{theorem}

\theoremstyle{definition}
\newtheorem{definition}[theorem]{Definition}

\theoremstyle{remark}

\numberwithin{equation}{section}

\newcommand{\ts}{\hspace{.11111em}}
\newcommand{\tts}{\hspace{.05555em}}

\newcommand{\Tcal}{\mathcal{T}}
\newcommand{\Ucal}{\mathcal{U}}

\newcommand{\R}{\mathbb{R}}
\newcommand{\T}{\mathbb{T}}

\DeclareMathOperator{\G}{\operatorname{\mathcal{G}}}  

\DeclareMathOperator{\area}{\operatorname{\mathsf{Area}}\tts}
\DeclareMathOperator{\vol}{\operatorname{\mathsf{Vol}}\tts}

\newcommand{\RP}{\mathbb{R}\mathbb{P}}

\DeclareMathOperator{\sys}{\operatorname{\mathsf{Sys}}\tts}

\DeclareMathOperator{\SR}{\operatorname{\mathsf{SR}}\tts}
\DeclareMathOperator{\FillVol}{\operatorname{\mathsf{Fill} \ts \mathsf{Vol}}\tts} 

\usepackage[pagewise]{lineno}

\newcommand*\patchAmsMathEnvironmentForLineno[1]{%
  \expandafter\let\csname old#1\expandafter\endcsname\csname #1\endcsname
  \expandafter\let\csname oldend#1\expandafter\endcsname\csname end#1\endcsname
  \renewenvironment{#1}%
     {\linenomath\csname old#1\endcsname}%
     {\csname oldend#1\endcsname\endlinenomath}}%
\newcommand*\patchBothAmsMathEnvironmentsForLineno[1]{%
  \patchAmsMathEnvironmentForLineno{#1}%
  \patchAmsMathEnvironmentForLineno{#1*}}%
\AtBeginDocument{%
\patchBothAmsMathEnvironmentsForLineno{equation}%
\patchBothAmsMathEnvironmentsForLineno{align}%
\patchBothAmsMathEnvironmentsForLineno{flalign}%
\patchBothAmsMathEnvironmentsForLineno{alignat}%
\patchBothAmsMathEnvironmentsForLineno{gather}%
\patchBothAmsMathEnvironmentsForLineno{multline}%
}

\begin{document}


\title{Systolic volume and complexity of 3-manifolds}

\author[L.~Chen]{Lizhi Chen}

\thanks{}

 \address{
\hspace*{0.155in}School of Mathematics and Statistics, Lanzhou University \newline
\hspace*{0.175in} Lanzhou 730000, P.R. China 
}

\email{lizhi@ostatemail.okstate.edu}

\thanks{The work is supported by the Fundamental Research Funds lzujbky-2017-26 for the Central Universities.}



\date{\today}

\begin{abstract}

In this paper, we prove that the systolic volume of a closed aspherical 3-manifold is bounded below in terms of complexity. Systolic volume is defined as the optimal constant in a systolic inequality. Babenko showed that the systolic volume is a homotopy invariant. Moreover, Gromov proved that the systolic volume depends on topology of the manifold. More precisely, Gromov proved that the systolic volume is related to some topological invariants measuring complicatedness. In this paper, we work along Gromov's spirit to show that systolic volume of 3-manifolds is related to complexity. The complexity of 3-manifolds is the minimum number of tetrahedra in a triangulation, which is a natural tool to evaluate the combinatorial complicatedness. 

\end{abstract}

\maketitle 

\tableofcontents

\section{Introduction}

Let $M$ be a closed aspherical manifold endowed with a Riemannian metric $\G$, denoted $(M, \G)$. The homotopy $1$-systole of $(M, \G)$, denoted $\sys \pi_1 (M, \G)$, is the shortest length of a noncontractible loop in $M$. Gromov \cite[Theorem 0.1.A.]{Gro83} proved that on a closed asphercial $n$-manifold $M$, every Riemannian metric $\G$ satisfies the following systolic inequality,
\begin{equation} \label{sys_Gromov}
 \sys \pi_1 (M, \G)^n \leqslant C_n \ts \vol_{\G} (M),
\end{equation}
where $C_n$ is a constant only depending on the manifold dimension $n$. Gromov \cite[Section 4.26.]{Gro07} further verified that the constant $C_n$ is related to the topology of the manifold $M$,
\begin{equation} \label{sys_ine_top}
 \sys \pi_1 (M, \G)^n \leqslant \left|\textbf{ Top }\right| \ts \vol_{\G} (M) ,
\end{equation}
where $\left| \textbf{Top} \right|$ is some measure of the topological complexity of $M$. In this paper, we work along Gromov's spirit to show that the constant $\left| \textbf{Top} \right|$ is dependent on complexity defined in terms of triangulation of a closed aspherical 3-manifold.


Examples of systolic inequalities of type $(\ref{sys_ine_top})$ include the case of closed hyperbolic surfaces $\Sigma_g$ of genus $g$,
\begin{equation*}
 \sys \pi_1(\Sigma_g, \G)^2 \leqslant C \ts \frac{\log^2{g}}{g} \area_{\G} (\Sigma),
\end{equation*}
where $C$ is a fixed constant. The constant $| \textbf{Top} |$ here depends on the genus $g$. A generalization to higher dimensional aspherical $n$-manifolds $M$ is realized by using simplicial volume,
\begin{equation} \label{sys_SimVol}
 \sys \pi_1 (M, \G)^n \leqslant C_n \ts \frac{\log^3{(1 + \|M\|)}}{\| M \|} \ts \vol_{\G}(M),
\end{equation}
where $C_n$ is a constant only depending on the dimension $n$. The simplicial volume $\| M \|$ is defined as Gromov norm of the fundamental class $[M]$, that is, 
\begin{equation*}
 \| M \| = \inf \sum_{i} |a_i|, 
\end{equation*}
where the infimum is taken over all cycles $\sum_{i} a_i \sigma_i$ representing the fundamental class $[M]$ in $H_n (M; \R)$. Moreover, the constant $| \textbf{Top} |$ in (\ref{sys_ine_top}) can be dependent on simplicial height $h(M)$, 
\begin{equation} \label{sim_height}
 \sys \pi_1(M, \G)^n \leqslant C_n \ts \frac{\exp{ \left( C_n^{\prime} \sqrt{ \log{h(M)} } \right) }}{h(M)} \vol_{\G} (M), 
\end{equation}
where $C_n$ and $C_n^{\prime}$ are constants only depending on the dimension $n$, see Gromov \cite[Section 3.C.3.]{Gro92} and \cite[Section 4.26.]{Gro07}. The simplicial height $h(M)$ of a closed manifold $M$ is defined to be the minimum number of simplices in a geometric cycle representing the fundamental class $[M]$. 

Moreover, Gromov \cite[Section 3.C.3.]{Gro92} also defined modified simplicial height $h_{+}(M)$ of a closed aspherical manifold $M$ as
\begin{equation*}
 \inf \sum_{i} |\lambda_i | ,
\end{equation*}
where the infimum is taken over all integral cycles $\sum_{i} a_i \sigma_i$ representing the fundamental class $[M]$. He further mentioned in \cite[Section 3.C.3.]{Gro92} that it is desirable to have an inequality of type similar to (\ref{sys_SimVol}) or (\ref{sim_height}). In this paper, we provide a partial answer to this problem in dimension three. Matveev \cite{Mat90} defined complexity of closed 3-manifolds. The complexity is a topological invariant equal to the minimum number of tetrahedra in a triangulation, if $M$ is a closed irreducible 3-manifold other than $S^3, \RP^3$ and the lens space $L(3, 1)$, see Section 3. Hence according to the above definitions, Gromov's modified simplicial height coincides with complexity, if the closed $3$-manifold $M$ is irreducible and not $S^3$, $\RP^3$ and $L(3, 1)$. In this paper, we show that the modified simplicial height, or complexity, satisfies inequality of the form like (\ref{sim_height}), if $M$ is a $3$-manifold described above.  

The optimal constant in a systolic inequality is usually called systolic volume, see \cite{Sa06, Sa07}.
\begin{definition}
The systolic volume of a closed aspherical $n$-dimensional manifold $M$, denoted by $\SR(M)$, is defined as
\begin{equation*}
 \inf_{\G} \frac{\vol_{\G}(M)}{\sys \pi_1(M, \G)^n},
\end{equation*}
where the infimum is taken over all Riemannian metrics $\G$ on $M$. 
\end{definition}
It is proved by Babenko \cite{Ba92} that the systolic volume $\SR(M)$ is a homotopy invariant. The main theorem of this paper is that the systolic volume of a closed aspherical 3-manifold has a lower bound depending on complexity.
\begin{theorem}  \label{main_thm}
 Let $M$ be a closed aspherical $3$-manifold. The systolic volume $\SR(M)$ satisfies
 \begin{equation}
  \SR(M) \geqslant C \ts \frac{c(M)}{\exp{ \left( C^{\prime} \ts \sqrt{ \log{ c(M) } } \right) }} ,
  \label{main_ine}
 \end{equation}
 where $C$ and $C^{\prime}$ are two constants which can be explicitly calculated. 
\end{theorem}

In the proof of Theorem \ref{main_thm}, we construct nerve of open cover of a smooth triangulation $K$ of $M$. The covering argument and method of nerves are used in Gromov \cite{Gro83} ( also see Guth \cite{Gu11} ). In terms of the construction of nerve, we can have a lower bound of systolic volume depending on the number of simplices in a polyhedron. The Alexandroff map from the triangulation $K$ to the nerve can be made to be simplicial, so that there is a comparison of the simplices in triangulation $K$ and the nerve. Hence we get the estimate of systolic volume in terms of complexity. 

The balls $B(x, r) \subset K$ in the cover used to construct nerve has the following property,
\begin{equation*}
 \vol ( B(x, r) ) \geqslant C \ts r^n
\end{equation*}
holds when $r \in (0, \frac{1}{2} \sys \pi_1 (K)) $, where $B(x, r)$ is the ball with center $x$ and radius $r$, and $C$ is a constant. In order to have such an estimate of volume of balls, we embed $K$ isometrically into the Banach space $L^{\infty}$. Here $L^{\infty}$ is the space of all bounded Borel functions on $K$, which appears in Gromov's proof of the systolic inequality, see \cite{Gro83}. Such an embedding is called Kuratowski embedding. Moreover, there is a finite approximation of the embedding. The estimate of volume of balls holds since there are isoperimetric inequalities in the Banach space $L^{\infty}$, see \cite{Gro83} and \cite{AmKa11}. 

This paper is organized as follows. Section 2 contains a brief review to the present research progress of the systolic volume. In Section 3, we introduce the complexity invariant of 3-manifolds. In Section 4, Gromov's work on Kuratowski embedding and isoperimetric inequality in Banach spaces are introduced. They are used in the proof of the main theorem of this paper. In Section 5, we explain how to construct cubic complexes and $\delta$-extensions in the Banach space. In Section 6, good open cover and the associated nerve is presented. Construction of the nerve is a main technique in the proof of main theorem. The proof of main Theorem \ref{main_thm} is given in Section 7.

\section{Systolic volume of aspherical manifolds}

Study of systolic volume occupies a main part in the research of systolic geometry, see \cite{CrK03} or \cite{Ka07}. Problems related to the systolic volume include calculating exact values, optimal metrics realizing systolic volume, topological properties related to systolic volume. Note that in literature, the systolic volume is sometimes called optimal systolic ratio ( see Croke and Katz \cite{CrK03}), or systolic constant ( see Brunnbauer \cite{Bru08}, Elmir and Lafontaine \cite{MiLa13}). 

Exact values of systolic volume of surfaces is known for $2$-torus $\SR(\T^2) = \frac{\sqrt{3}}{2}$, real projective plane $\SR(\RP^2) = \frac{2}{\pi}$ and Klein bottle $\SR(\RP^2 \# \RP^2) = \frac{2\sqrt{2}}{\pi}$. Moreover, among all closed manifolds with nonzero systolic volume, we currently only know exact values of systolic volume for these three examples.  

In $2$-dimension, a closed non-simply connected surface $\Sigma$ satisfies
\[ \SR(\Sigma) \geqslant \frac{2}{\pi}, \]
where equality holds if $\Sigma$ is a real projective plane $\RP^2$, see Croke and Katz \cite{CrK03}. Morevoer, if $\Sigma_g$ is a closed orientable hyperbolic surface of genus $g$,
\begin{equation}
 \pi \frac{g}{\left( \log{g} \right)^2} \lesssim \SR(\Sigma_g) \lesssim \frac{9\pi}{4} \frac{g}{\left( \log{g} \right)^2 }
 \label{sr_2}
\end{equation}  
holds when the genus $g$ is large enough, see \cite{KaSa05}. More estimates of systolic volume for surfaces can be found in Croke and Katz \cite[Figure 2.1.]{CrK03}.  

In higher dimensions, Gromov \cite{Gro83} proved that
\begin{equation*}
 \SR(M) \geqslant \frac{1}{ \left( 6 (n+1) n^n \sqrt{(n + 1)!} \right)^n }
\end{equation*} 
holds for aspherical $n$-dimensional manifolds $M$. Wenger's work \cite{Wen08} made an improvement,
\begin{equation*}
 \SR(M) \geqslant \frac{1}{ \left( 6 \cdot 27^n n! \right)^n }.
\end{equation*}
Guth \cite{Gu10} further improved the lower bound estimate of systolic volume,
\begin{equation*}
 \SR(M) \geqslant \frac{1}{(8n)^n},
\end{equation*}
where $M$ is $n$-torus $\T^n$, or real projective $n$-space $\RP^n$, or $n$-manifolds with nice cohomology properties. 

For the relation between systolic volume and other topological invariants, Gromov also showed $\left| \textbf{Top} \right|$ in (\ref{sys_ine_top}) is dependent on Betti numbers $b(M)$,
\begin{equation*}
 \SR(M) \geqslant C_n \ts \frac{b(M)}{\exp{\left( C_n^{\prime} \sqrt{\log{b(M)}} \right)}} , 
\end{equation*}
where $C_n$ and $C_n^{\prime}$ are constants only depending on the dimension $n$. Sabourau extend this result to the connected sum of aspherical manifolds,
\begin{equation} \label{Sa07_connected}
 \SR(\#_k M) \geqslant C_n \ts \frac{k}{\exp{( C_n^{\prime}  \sqrt{\log{k}} )}},
\end{equation}
where $\#_k M$ stands for the connected sum of $k$ copies of the manifold $M$, and $C_n$ and $C_n^{\prime}$ are two constants only depending on the dimension $n$ of the manifold $M$. Babenko and Balacheff \cite{BabBal05} proved an upper bound to the connected sum,
\begin{equation*}
 \SR(\#_k M) \leqslant C \ts \frac{k \ts \sqrt{\log{(\log{k} )}}}{\sqrt{\log{k}}},
\end{equation*}
with $C$ a constant only depending on the manifold $M$. In addition, they conjectured that the systolic volume satisfies 
\begin{equation*} 
 \SR(\#_k M) \geqslant C \ts \frac{k}{\log^n{k}},
\end{equation*}
where $\#_k M$ is the connected sum of $k$ copies of the aspherical $n$-manifold $M$, with $n \geqslant 3$, see \cite{BabBal05}. Sabourau \cite{Sa06} and Brunnbauer \cite{Bru08} showed that the systolic volume of aspherical $n$-manifolds has lower bound in terms of minimal entropy $\lambda (M)$,
\begin{equation*}
 \SR(M) \geqslant C_n \frac{\lambda(M)}{\log^n{(1 + \lambda(M))}},
\end{equation*}
which is another example to the type of inequality (\ref{sys_ine_top}). 

\section{Complexity of $3$-manifolds}

The complexity of 3-manifolds can be considered as a natural topological invariant to evaluate how complicated the manifold is. 

Matveev \cite{Mat90} defined complexity of a compact 3-manifold $M$ as the minimal number of vertices in a special spine of $M$. For a closed irreducible $3$-manifold other than $S^3$, $\RP^3$ and the lens space $L(3, 1)$, the complexity is equal to the minimum number of tetrahedra in a triangulation. In this paper, we define the complexity of closed aspherical 3-manifolds to be minimum number of tetrahedra in a triangulation.  

A triangulation of a $3$-manifold $M$, denoted $(\mathcal{T}, h)$, is a simplicial complex $\mathcal{T}$ with a homeomorphism $h: |\mathcal{T}| \to M$, where $| \mathcal{T} |$ stands for the union of all simplices of $\mathcal{T}$. The union $|\mathcal{T}|$ is called the underlying polyhedron of the triangulation. 
\begin{definition}[Pseudo-simplicial triangulation, see \cite{JaRu03}]
 A pseudo-simplicial triangulation $\mathcal{T}$ of a $3$-manifold $M$ contains
 \begin{enumerate}
  \item a set $\Delta = \{\tilde{\Delta}_i \}$ of disjoint collection of tetrahedra,
  \item a family $\Phi$ of isomorphisms pairing faces of the tetrahedra in $\Delta$ so that if $\phi \in \Phi$, then $\phi$ is an orientation-reversing affine isomorphism from a face $\tilde{\sigma}_i \in \tilde{\Delta}_i$ to a face $\tilde{\sigma}_j \in \tilde{\Delta}_j$, possibly $i = j$.
 \end{enumerate}
 Denote by $\Delta / \Phi$ the space obtained from the disjoint union of the $\tilde{\Delta}_i$ by setting $x \in \tilde{\sigma}_i$ equal to $\phi(x) \in \tilde{\sigma}_j$, with the identification topology. The manifold $M$ is homeomorphic to $|\mathcal{T}| = \Delta / \Phi$.  
\end{definition} 
In a pseudo-simplicial triangulation, two faces of the same tetrahedra possibly are identified. 

\begin{definition}
 Let $M$ be a closed irreducible $3$-manifold. The complexity of $M$, denoted $c(M)$, is the minimal number of tetrahedra in a pseudo-simplicial triangulation. 
\end{definition}
The complexity $c(M)$ has the following finiteness property and additivity property. 
\begin{proposition}[Matveev {\cite[Theorem A, B]{Mat90}}]
 \begin{enumerate}
  \item For any integer $k$, there exist only a finite number of distinct closed irreducible orientable $3$-manifolds of complexity $k$.
  \item The complexity of the connected sum of compact $3$-manifolds is equal to the sum of their complexities. 
 \end{enumerate}
 \label{compl_3}
\end{proposition}

The complexity of a 3-manifold measures how complicated a combinatorial description of the manifold must be, see \cite{Mat90}. It is difficult to determine the complexity of 3-manifolds. There are various investigations of upper bounds estimate and lower bounds estimate to complexity. Moreover, exact values of complexity are calculated for some family of infinitely many 3-manifolds. For example, Matveev et al. \cite{MPV09} proved two-sided bounds for some families of hyperbolic 3-manifolds. Anisov \cite{Ani05} calculated exact values of complexity for a family of infinitely many $3$-manifolds with Sol geometry. More results are included in a series of papers by Jaco, Rubinstein and Tillmann, see \cite{JaRuTill09, JaRuTill11, JaRuTill13}. Result of this paper can be seen as one more way to the study of complexity.  

\section{Kuratowski embedding and isoperimetric inequality}

A manifold with a distance function can be embedded into a metric space by Kuratowski embedding. Gromov \cite{Gro83} proved the systolic inequality (\ref{sys_Gromov}) by using Kuratowski embedding and generalized isoperimetric inequality in Banach spaces. 

Let $M$ be a closed connected $n$-dimensional manifold. The set of all bounded Borel functions $f$ on $M$, denoted $L^{\infty}(M)$, is a Banach space with the norm
\[ \|f\|_{\infty} = \sup_{x \in M} | f(x) |, \; f \in L^{\infty}(M). \] 
Given a distance function $\mathrm{dist}_{M}$, the manifold $M$ becomes a metric space. Define the map $K: M \to L^{\infty}(M)$ by setting
\[ K(x)(w) = dist_M (x, w) \]
for any $w \in M$.
\begin{lemma}[Gromov 1983, \cite{Gro83} Section 2.]
For any $v, w \in M$, the following equality holds,
\begin{equation}
 \text{dist}(v, w) = \| K(v) - K(w) \|_{\infty} .
\end{equation}
 \label{Ku}
\end{lemma}
Lemma \ref{Ku} implies that $K: M \to L^{\infty}(M)$ is an embedding. 
\begin{definition}
 The map $K: M \to L^{\infty}(M)$ is called Kuratowski embedding. 
\end{definition} 

The Kuratowski embedding $K: M \to L^{\infty}(M)$ can be approximated by a finite dimensional embedding. Let $M_0 \subset M$ be an $\varepsilon$-net. Then there exists a $(1+ \varepsilon)$ bi-Lipschitz finite dimensional embedding of $M$, see \cite[Proposition 5.1.]{AmKa11}. 

Gromov defined filling volume of a manifold in \cite{Gro83}. Let $[M]$ be the fundamental class of $M$. For a closed manifold $M$ with a distance function defined on it, the filling volume of a submanifold $V \subset M$, denoted $ \FillVol ( V \subset M ) $,
 is defined to be the lower bound of volumes of the cycles $z$ with $\partial z = [V]$. 
\begin{definition}
 Let $M$ be a closed manifold endowed with a distance function. The filling volume of $M$, denoted $\text{Fill Vol}(M)$, is defined as
 \[ \FillVol (M \subset L^{\infty}(M)).  \]
\end{definition}

Given a Riemannian manifold $(M, \G)$, we have a distance function $dist_{\G}$ induced by the Riemannian metric. Gromov showed the following isoperimetric
inequality in $L^{\infty}(M)$, which is a generalization to Federer and Fleming's isoperimetric inequality in Euclidean space, see \cite{Gro83} or \cite{Gu06} for a reference. 
\begin{theorem}[Gromov \cite{Gro83}]
 Let $M$ be a closed Riemannian manifold. For any $n$-dimensional cycle $z \in L^{\infty}(M)$, 
  \[ \FillVol (z) \leqslant C_n \vol (z)^{\frac{n+1}{n}}, \]
  where $C_n$ is a constant only depending on the cycle dimension $n$. 
\end{theorem}

\section{Cubical complexes and $\delta$-extension}

In this section, we show a tool which is important in the construction of good balls. We will use the cover of such open good balls to construct nerve, see Section 6 in the following. Materials contained in this section are from Gromov \cite[Section 6.2.]{Gro83}. 

A cubical complex is similar to the simplicial complex. The difference is that we take cubes instead of simplexes. Given a manifold $M$ with distance function $\text{dist}_M$, we define a cubic complex in terms of Kuratowski embedding.   

Let $\delta > 0$ be a given constant. The standard $\delta$-cube in $L^{\infty}(M)$ is the set of functions
\[ \{ \varphi(x) | 0 \leqslant \varphi(x) \leqslant \delta, \text{for all} \; x \in M . \} \]
\begin{definition}
 For a given $\delta > 0$, we define a cubical $\delta$-complex to be a metric space $K$, which is partitioned into isometric image of standard $\delta$-cubes, such that any two cubes meet at a face.  
\end{definition}
In a cubical $\delta$-complex, there holds the following isoperimetric inequality. 
\begin{theorem}[Gromov \cite{Gro83}, Section 6.2.B.]
 For a given integer $n = 1, \cdots , $ there exist two positive constants $\mu_n$ and $C_n^{\prime}$, such that every singular $n$-dimensional cycle $z$ in any
 cubical $\delta$-complex $K$ with
 \[ \vol(z) \leqslant \mu_n \delta^n  \]
 bounds a chain $c \subset K$, and
 \begin{equation} \label{iso_cube}
  \vol (c) \leqslant C_n^{\prime} \left[ \vol (z) \right]^{(n+1)/ n} , 
 \end{equation}
such that $c$ is contained in the $\varepsilon^{\prime}$-neighborhood of $z$ for $\varepsilon^{\prime} = C_n^{\prime} \left( \vol z \right)^{1/n}$. 
\end{theorem}

Assume that $M_0 \subset M$ is a finite subset with $N$ elements. The Banach space $L^{\infty}(M_0)$ is isometric to an $N$-dimensional space $\ell^{N}_{\infty}$, where the norm on $\ell^{N}_{\infty}$ is 
\[ \| (x_1, x_2, \cdots , x_N) \|_{\infty} = \sup_{1 \leqslant i \leqslant N} |x_i| . \]
We define a map $I_0: M \to L^{\infty}(M_0)$ as follows,
\[ x \mapsto \varphi_x \]
with $ \varphi_x (w) = \min \{ \delta, \text{dist}_M (x, w) \} $ for any $w \in M_0$.

The finite dimensional Banach space $L^{\infty}(M_0)$ can be decomposed into $\delta$-cubes by using the following hyperplanes
\[ \{ \varphi | \varphi (x) = m \delta , \; x \in M_0 \}, \]
where $m \in \mathbb{N}$ are integers. For each point $x \in L^{\infty}(M_0)$, after specifying a $\delta$-complex structure, there is a unique cube containing $x$, denoted $\square (x)$. 
Assume that $K \subset L^{\infty}(M_0)$. 
For a sufficiently small number $\varepsilon > 0$, we define a Lipschitz map $R_{\varepsilon}: K \to K$ with the Lipschitz constant $1/ ( \frac{1}{2} \delta - \varepsilon)$, which contracts 
each $\varepsilon$-neighborhood of each subcomplex $K_0 \subset K$ onto $K_0$. The map $R_{\varepsilon}$ can be defined in terms of the $1$-dimensional Lipschitz map 
$r_{\varepsilon}: [0, 1] \to [0, 1]$:
\[ R_{\varepsilon} (x_1, x_2, \cdots , x_N) = (r_{\varepsilon}(x_1), r_{\varepsilon}(x_2), \cdots , r_{\varepsilon}(x_N) ). \] 
\begin{definition}
 A $\delta$-extension of $M$, denoted by $K_{\delta}(M)$, is the union
  \begin{equation*}
   \cup_{x \in M} \square (x),
  \end{equation*}
  where $\square(x)$ stands for the unique $\delta$-cube with minimal dimension and containing the point $x^{\prime} = R_{\varepsilon} \circ I_0 (x)$. 
\end{definition}

Let the map $J: M \to K_{\delta}(M)$ be defined as $J = R_{\varepsilon} \circ I_0 $. 
\begin{proposition}
 The image $M^{\prime} = J(M)$ of the map $J$ is homeomorphic to $M$. 
\end{proposition}
\begin{proof}
 The finite dimensional approximation $K_0: M \to L^{\infty}(M_0)$ of the Kuratowski embedding is bi-Lipschitz, see \cite[Proposition 5.1.]{AmKa11}, so that is an embedding. By a slight modification of the proof there, we can see that 
 the map $I_0: M \to L^{\infty}(M_0)$ is a local embedding. Since we let $\varepsilon \leqslant \delta$, and each $\delta$-neighborhood of point $x \in M$ is contractible, so that 
image $R_{\varepsilon}(I_0 (M) )$ of the Lipschitz map $R_{\varepsilon}$ is homeomorphic to $I_0 (M)$.
 Hence $J$ is an embedding.
\end{proof}

Note that the map $J$ is Lipchitz with a Lipschitz constant $\frac{\delta}{\delta - 2 \varepsilon}$. And for the volume, we have
\begin{equation}
 \vol (M^{\prime}) \leqslant \left( \frac{\delta}{\delta - 2 \varepsilon} \right)^n \vol (M) .
 \label{vol_cube}
\end{equation}
The volume of $M^{\prime}$ is defined in terms of the norm in finite Banach space $L^{\infty} (M_0)$, see Gromov \cite[Section 4]{Gro83}. 

\section{Good cover and nerve}

In this section, we introduce the notion of nerve of cover of open balls. The estimate of using nerve of covers can be found in Guth \cite{Gu11} and Gromov \cite[Section 5 and Section 6]{Gro83}. 

Let $X$ be a topological space. 
\begin{definition}
 The nerve $N(\mathcal{U})$ of an open cover $\mathcal{U} = \{U_{\lambda}\}_{\lambda \in \Lambda}$ is a simplicial complex, which satisfies:
 \begin{enumerate}
  \item there is a vertex $v_{\lambda}$ for each open subset $U_{\lambda} \subset \mathcal{U}$; 
  \item there is a $k$-simplex spanned by vertices $v_{\lambda_1}, v_{\lambda_2}, \cdots , v_{\lambda_k}$, if the corresponding subsets have nonempty intersection, $U_{\lambda_1} \cap U_{\lambda_2} \cap \cdots \cap U_{\lambda_k} \neq \emptyset$. 
 \end{enumerate}
\end{definition}
Denote the underlying polyhedron of $N(\Ucal)$ by $N$.

\begin{definition}
 An open cover $\Ucal$ of $X$ is called a good cover, if each subset $U \in \Ucal$ is contractible.
\end{definition}
\begin{theorem}[Nerve Theorem, Hatcher \cite{Ha02} Corollary 4G.3.]
 If $X$ is a paracompact topological space, and $N$ is the underlying polyhedron of the nerve $N(\Ucal)$ of a good cover $\Ucal$ on $X$, then $N$ is homotopy equivalent to $X$.
 \label{nerve_thm}
\end{theorem}

Given a closed $n$-dimensional manifold $M$. Let $\Ucal$ be an open good cover on $M$. Then the underlying polyhedron $N$ of the nerve $N(\Ucal)$ is homotopy equivalent to $M$, according to the above nerve theorem. From the manifold $M$ to the underlying polyhedron $N$ of the nerve $N(\Ucal)$, there exists canonical Alexandroff map $A: M \to N(\Ucal)$, see Karoubi and Weibel \cite[Proposition 2.1.]{KaWe16} or references therein. Assume that $\{ a_i \}$ is a partition of unity subordinate to the open cover $\Ucal$, then the Alexandroff map $A$ takes each point $x \in M$ to a point with barycentric coordinate $a_i (x) \in N$. 

Moreover, if $K$ is the underlying polyhedron of a triangulation of $M$, the Alexandroff map from the manifold to the nerve can be made into simplicial. Let $\Ucal$ be an open good cover of open balls over $K$. Moreover, assume that each vertex of the underlying polyhedron $K$ is the center of some open ball in the cover $\Ucal$. Hence the Alexandroff map is simplicial, if every ball $B \in \Ucal$ with center a vertex of $K$ does not contain any other vertices of $K$. We take $\varphi$ to be the Alexandroff map from $K$ to $N$.    
\begin{lemma}
 For the underlying polyhedron $K$ of each triangulation of $M$, there exists a nerve $N(\Ucal)$ of an open good cover $\Ucal$ of $M$, whose underlying polyhedron $N$ has more vertices than $K$. 
\end{lemma}
\begin{proof}
 As mentioned above, there exists a simplicial map $\varphi: K \to N$ for the nerve $N(\Ucal)$ of the open good cover $\Ucal$ over $K$. Furthermore, with respect to different vertices $v_1, v_2 \in K$, the images $\varphi (v_1)$ and $\varphi_2(v_2)$ are distinct due to our previous assumption. Hence in the underlying polyhedron $N$ of the nerve $N(\Ucal)$, there will be more vertices than $K$.  
\end{proof}  

For a closed Riemannian $n$-manifold $(M, \G)$, there exists an open good cover $\Ucal = \{ B_{x_i}(r_i) \}$, where $B_{x_i}(r_i)$ is an open ball with center $x_i$ and radius $r_i$. 
\begin{definition}
 Let $R_0 = \frac{1}{10} \sys \pi_1(M, \G)$. A ball $B_x(r)$ with radius $0 < r \leqslant R_0 $ is called admissible, if there exists a constant $\alpha \geqslant 1$ so that 
 \begin{enumerate}
  \item \[ \vol_{\G}(B_x(5 r)) \leqslant \alpha \vol_{\G}(B_x(r)), \] 
  \item \[ \vol_{\G}(B_x (5 r^{\prime})) \geqslant \alpha \vol_{\G}(B_x (r^{\prime})) \]
   holds for any $r< r^{\prime} \leqslant R_0$.
 \end{enumerate} 
\end{definition}
\begin{proposition}
 For any point $x \in M$, there exists a radius $ 0 < r \leqslant R_0 $, such that $B_x (r)$ is an admissible ball.
\end{proposition}
\begin{proof}
 The existence of admissible ball is true, since when $r$ is sufficiently small, the ball $B_x(r)$ is close to an Euclidean ball. 
\end{proof}

Let $\widetilde{\Ucal} = \{ B_{x_j}(r_j) \}_{j = 1}^N$ be a maximal system of disjoint admissible balls in $M$, so that 
\begin{enumerate}
 \item  $B_{x_1}( r_1)$ is the admissible ball of the largest radius $r_1$, and $r_1 \geqslant r_2 \geqslant \cdots r_N$;
 \item $\Ucal = \{ B_{x_j} (2 r_j) \}$ is an open cover of the manifold $M$. 
\end{enumerate}
It is easy to show that $\Ucal$ is a good cover. Therefore, the associated nerve $N(\Ucal)$ has its underlying polyhedron $N$ homotopy equivalent to $M$. 
Moreover, we have a simplicial map $\varphi: K \to N$ from the underlying polyhedron $K$ of a triangulation of the manifold $M$ to N. 

\section{Proof of the main theorem}

In this section, we prove the main theorem of this paper (Theorem \ref{main_thm}). The proof is separated into two steps.

\subsection{Volume of balls}

Let $M$ be a closed $n$-dimensional aspherical manifold. There exists a sequences of Riemannian metrics $\{ \G_k \}_{k = 1}^{\infty}$, so that 
 \begin{equation*}
  \lim_{k \to \infty} \frac{\vol_{\G_k}(M)}{\sys \pi_1 (M, \G_k)^n} = \SR (M) . 
 \end{equation*}
Normalize the metric $\G_k$ by scaling to let $\sys \pi_1 (M, \G_k) = 1$, so that we have
\begin{equation*}
 \lim_{k \to \infty} \vol_{\G_k} (M) = \SR(M). 
\end{equation*}

Denote by $M_k$ the Riemannian manifold $(M, \G_k)$. Let $\delta > 0$ be a given small positive number. We choose a sufficiently dense finite subset $M_k^{\circ} \subset M_k$. For example, we 
can choose $M_k^{\circ}$ to be an $\varepsilon$-net. For each $k$, there is a $\delta$-extension $K_{\delta}(M_k) \subset L^{\infty} (M_k^{\circ})$, see Section 5. The Riemannian manifold can be embedded into $K_{\delta}(M_k)$ by the Lipschitz map $J$ described in Section 5. In order to specify the distance $dist_k$ induced by the Riemannian metric $\G_k$ on $M_k$, we use the notation $J_k: M_k \to K_{\delta}(M_k)$ to denote the Lipschitz embedding.  

Let $M_k^{\prime}$ be the image of $M_k^{\prime}$ under the Lipschitz map $J_k: M_k \to K_{\delta}(M_k)$. Recall that the length and the volume in the Banach space $L^{\infty}(M)$ is defined in terms of the norm, see Section 4 or Gromov \cite{Gro83}. The volume of $M_k^{\prime}$ satisfies the inequality (\ref{vol_cube}). For the purpose of convenience in the following argument, we vary the Riemannian metric $\G_k$ on $M$ to let the following fact holds,
\begin{equation*}
 \lim_{k \to \infty} \vol (M_k^{\prime}) = \SR(M). 
\end{equation*}

\begin{lemma} \label{volume_ball}
When $k$ is sufficiently large, the volume of every ball $B_w(R)$ in $M_k^{\prime}$ satisfies 
  \begin{equation}
   \vol (B_{w}(R)) \geqslant A_n R^n,  
   \label{ball}
  \end{equation}
where $A_n$ is a constant only depending on the dimension $n$. 
\end{lemma} 

\begin{proof}
 We prove this Lemma by choosing a minimizing sequence of cycles. Such technique is used by Gromov in \cite[Section 4.3. $C^{\prime\prime}$]{Gro83}, also see \cite[Section 6.4.B.]{Gro83}. 
 
 Since the $\delta$-extension $K_{\delta}(M_k)$ is compact, we can assume that there is a subsequence of $\{ M_k^{\prime} \}$ with limit $\widetilde{M^{\prime}}$ of Gromov-Hausdorff convergence. For convenience, we still denote the subsequence by $M_k^{\prime}$. We may further assume that $\widetilde{M^{\prime}}$ is minimal, that is, no proper subset of $\widetilde{M^{\prime}}$ is the limit of any other minimizing sequence. The limit $\widetilde{M^{\prime}}$ of Gromov-Hausdorff convergence here serves for a formal device of the minimizing sequence, see Gromov \cite[4.3.$C^{\prime\prime}$]{Gro83}

In the cubic complex $K_{\delta}(M_k^{\prime})$, there holds the isoperimetric inequality (\ref{iso_cube}), see Section 5.   
 Hence by applying an argument analogous to Gromov \cite[Theorem 4.3.$C^{\prime\prime}$.]{Gro83}, we have
 \begin{equation*}
  \vol (B_{w} (R) ) \geqslant A_n R^n
 \end{equation*}
 holds for every $w \in M_k^{\prime}$ when $k$ is sufficiently large. 
\end{proof}

\subsection{Triangulation and nerve of open cover}

Assume that $k$ is large enough, so that the estimate (\ref{ball}) holds. We construct a nerve $N(\Ucal)$ for an open cover $\Ucal$ of $M_k^{\prime} $. 

Let $R_0 = \frac{1}{10} \sys \pi_1 ( M_k^{\prime} ) $. On $ M_k^{\prime}  \subset K_{\delta}(M_k)$, we choose a system of disjoint admissible balls $ \overline{\Ucal} = \{ B_{x_j} (r_j) \}_{j = 1}^{m}$, see Section 5 for a reference of admissible balls. Moreover, assume that the system $\overline{\Ucal}$ is maximal, and satisfying: 
\begin{enumerate}
 \item $\displaystyle \Ucal = \{ B_{x_j} (2 r_j) \}_{j = 1}^{m}$ is an open cover of $ M_k^{\prime} $; 
 \item each open ball $B_{x_j} (2 r_j)$ does not contain any other open balls in $\Ucal$. 
\end{enumerate}
 
Since the open cover $\Ucal$ is a good cover, the nerve $N(\Ucal)$ is homotopy equivalent to $ M_k^{\prime} $, so that homotopy equivalent to the manifold $M$. Given a smooth triangulation of the manifold $M$, denote by $K$
the underlying polyhedron. In construction of admissible balls in $\overline{\Ucal}$, we assume that for each vertex $v \in K$, the corresponding point in $  M_k^{\prime} $ is chosen to be the center of an admissible ball. Then we have a simplicial map 
$\varphi: K \to N$, where $N$ is the underlying polyhedron of the nerve $N(\Ucal)$. Therefore we proved the following result.
 \begin{proposition} \label{trian_nerve}
  For each smooth triangulation $\mathcal{T}$ of an aspherical manifold $M$, there exists an open cover so that the underlying polyhedron of the associated nerve has more vertices than the triangulation $\Tcal$. 
 \end{proposition}

\vskip 10pt
 

Now we show an estimate with respect to the systolic volume. Let the systolic volume be denoted by $S$ in the following.
 \begin{proposition}
  For each closed aspherical $n$-manifold $M$, there exists an open cover $\Ucal$, so that the following inequality holds for the number $T$ of simplices in the associated nerve and the systolic volume $S$, 
  \begin{equation} \label{eq_asphe}
   T \leqslant B_n \ts S \ts \exp{ ( B_n^{\prime} \ts \sqrt{\log{S}} \ts )} ,
  \end{equation} 
  where $B_n$ and $B_n^{\prime}$ are two constants only depending on the dimension $n$. 
 \end{proposition}

\begin{proof}
The covering argument in the proof is the same as in Gromov \cite[Theorem 5.3.B.]{Gro83}. Another important reference is Gromov \cite[Section 6.4.C.]{Gro83}. 
 
Denote by $V_k$ the volume $\vol ( M_k^{\prime}  )$. Let $w \in M_k^{\prime} $ and $B_w(R) \subset M_k^{\prime} $ be an admissible ball. Choose the radius $R$ so that
 \begin{equation*}
  5^{- k} R_0 \leqslant R \leqslant 5^{-k + 1} R_0 , 
 \end{equation*}
 where $k \geqslant 1$ is an integer. Recall that $R_0 = \frac{1}{10} \sys \pi_1 ( M_k^{\prime} )$. 
 Since we have
  \begin{align*}
    V_k & \geqslant \vol \left( B_w  \left( \frac{1}{2} \sys \pi_1 ( M_k^{\prime} ) \right) \right) \\
    & = \vol (B_w (5R_0) ) \\ 
   & \geqslant \alpha  \vol (B_w (R_0)) \\
   & = \alpha \vol \left( B_w \left( 5^{k - 1} \cdot 5^{-k + 1} R_0 \right) \right) \\
    & \;\; \vdots \\
    & \geqslant \alpha^{k} \vol (B_w ( 5^{-k + 1} R_0)) \\
    & \geqslant \alpha^k \vol (B_w (R)) \\
    & \geqslant \alpha^k A_n R^n \\
    & \geqslant \left( \alpha 5^{-n} \right)^k A_n R_0^n ,
  \end{align*}
  so that
   \begin{equation*}
    k \leqslant \frac{\log{V_k} - \log{R_0^n} - \log{A_n} }{ \log{\alpha}  - n }  . 
   \end{equation*}
  Moreover, we choose the constant $\alpha$ to let
   \begin{equation}
    \log{\alpha} - n = \sqrt{\log{S_k} -  \log{A_n}  } , 
    \label{alpha}
   \end{equation}
  where we use $S_k$ to denote the systolic ratio $\frac{V_k}{R_0^n}$. Let $k$ be the greatest integer satisfying the equality (\ref{alpha}).
  
  Let $ \{ B_{w_j}(R_j) \}_{j = 1}^m$ be a maximal system of mutually disjoint admissible open balls, with $R_j \in (0, R_0]$ and satisfying
  \begin{enumerate}
   \item $R_1 \geqslant R_2 \geqslant \cdots \geqslant R_m$;
   \item $\Ucal = \{ B_{w_j} (2 R_j) \}_{j = 1}^{m}$ is an open cover of $M_k^{\prime}$. 
  \end{enumerate}
  

Let $N$ be the nerve associated to this cover. Denote by $m_j$ the number of balls intersecting with $B_{w_j}(R_j)$. Denote by $P$ the number of pairwise intersections in the open cover $\Ucal$.
Then we have
  \begin{align*}
    V_k & =  \vol (  M_k^{\prime}  ) \\
    & \geqslant \sum_{j = 1}^m \vol (B_{w_j} (R_j) ) \\
    & \geqslant \sum_{j = 1}^m \alpha^{-1} \vol (B_{w_j} ( 5 R_j) )  \\
    & \geqslant \alpha^{-1} \sum_{j=1}^m \sum_{\ell = 1}^{m_{j}} \vol (B_{w_{\ell}} (R_{\ell})  )  \\ 
    & \geqslant \alpha^{-1} \sum_{j = 1}^m \sum_{\ell = 1}^{m_j} A_n R_{\ell}^n \\
    & \geqslant \alpha^{-1} A_n \left( 5^{- k} R_0 \right)^n \ts P . \\
  \end{align*}
 Hence we have
  \begin{align*}
   P & \leqslant \frac{V_k}{R_0^n} \, \alpha \, A_n^{-1} \, 5^{nk} \\
    & = 5^n A_n^{-1} \, S_k \, 5^{(n + 1) \sqrt{\log{S_k - \log{A_n}}}} \\ 
    & \leqslant B_n \ts S_k \ts \exp{\left( B_n^{\prime} \ts \sqrt{\log{S_k}} \right)} ,
  \end{align*} 
 where $B_n = 5^n \ts A_n^{-1}$, $B_n^{\prime} = (n+1) \log{5}$ are constants only depending on $n$. 
 
 Denote by $T_i$ the number of $i$-simplices in the underlying polyhedron $N$ of the nerve $N(\Ucal)$, so that $T_2 = P$, and $T = \sum_{i} T_i$. 
 Since $T_0 \leqslant 2P$, and $T_i \leqslant P$ with $i \geqslant 2$, the estimate (\ref{eq_asphe}) also holds for any $T_i$ with $i \geqslant 0$, thus holds for $T$.
 
 If we let $k \to \infty$, then the estimate (\ref{eq_asphe}) is obtained. 
 \end{proof}
 
The inequality (\ref{eq_asphe}) implies that
 \begin{equation}
  \SR(M) \geqslant C_n \frac{T}{ \exp{ \left( C_n^{\prime}  \sqrt{\log{T}}  \right) } } ,
 \end{equation}
 where $C_n$ and $C_n^{\prime}$ are two positive constants only depending on the dimension $n$.
 
 Let $M$ be a closed aspherical $3$-manifold. In terms of Proposition \ref{trian_nerve}, for each triangulation of $M$, there exists an open cover from maximal system of disjoint admissible open balls, so that the number $t$ of tetrahedra in this triangulation is no more than the number $T_3$ of tetrahedra in the nerve. Hence we have
 \begin{equation*}
  \SR(M) \geqslant C \ts \frac{t}{ \exp{ \left( C^{\prime} \ts  \sqrt{\log{t}}  \right) } },
 \end{equation*} 
 where $C$ and $C^{\prime}$ are two positive constants. After taking infimum on $t$, we get the lower bound estimate of systolic volume in Theorem \ref{main_thm}.

\end{document}